\documentclass[12pt]{article}
\usepackage{fullpage}
\usepackage{amsmath,amssymb,amsfonts,amsthm}
\usepackage[all,arc,knot,poly]{xy}
\usepackage{enumerate}
\usepackage{mathrsfs}
\usepackage{amsxtra}
\usepackage{needspace}
\usepackage{graphicx}

\DeclareMathOperator{\Spec}{\operatorname{Spec}}
\DeclareMathOperator{\sgn}{\operatorname{sgn}}
\DeclareMathOperator{\Loc}{\operatorname{Loc}}
\DeclareMathOperator{\Hom}{\operatorname{Hom}}
\DeclareMathOperator{\Conf}{\operatorname{Conf}}
\DeclareMathOperator{\Aut}{\operatorname{Aut}}
\DeclareMathOperator{\im}{\operatorname{im}}

\newenvironment{step}[1][]{
\medskip
\emph{Step #1.}
}
{\medskip}

\theoremstyle{definition}
\newtheorem{theorem}{Theorem}[section]
\newtheorem{definition}[theorem]{Definition}
\newtheorem{lemma}[theorem]{Lemma}
\newtheorem{proposition}[theorem]{Proposition}

\title{The cluster symplectic double and \protect\\ moduli spaces of local systems}

\author{Dylan G.L. Allegretti}

\date{}

\begin{document}

\maketitle

\begin{abstract}
We prove a conjecture of Fock and Goncharov which provides a birational equivalence of a cluster variety called the cluster symplectic double and a certain moduli space of local systems associated to a surface.
\end{abstract}

\tableofcontents

\section{Introduction}

Cluster varieties are a class of geometric objects introduced by Fock and Goncharov as part of their study of higher Teichm\"uller theory~\cite{IHES, dilog}. In the original papers on the subject, Fock and Goncharov considered two types of cluster varieties. The first of these is known as the cluster $K_2$-variety or $\mathcal{A}$-variety. It is defined using the mutation formulas from Fomin and Zelevinsky's theory of cluster algebras~\cite{FZI}. The second type of cluster variety is known as the cluster Poisson variety or $\mathcal{X}$-variety. It possesses a natural Poisson structure~\cite{IHES, dilog}.

Cluster varieties appear naturally in several areas of mathematics and mathematical physics. As shown in~\cite{IHES}, these objects arise as certain moduli spaces of local systems associated to a compact oriented surface $S$ with boundary and finitely many marked points on the boundary. More precisely, Fock and Goncharov considered two moduli spaces, denoted $\mathcal{X}_{PGL_m,S}$ and $\mathcal{A}_{SL_m,S}$. The first of these spaces parametrizes $PGL_m$-local systems on the surface $S$ together with an additional datum called a framing. The second space parametrizes objects called decorated twisted $SL_m$-local systems on~$S$. Both of these moduli spaces will be defined more precisely below. One of the main results of~\cite{IHES} states that the space $\mathcal{X}_{PGL_m,S}$ is birational to a cluster Poisson variety, while $\mathcal{A}_{SL_m,S}$ is birational to a cluster $K_2$-variety.

In subsequent work, Fock and Goncharov introduced a third type of cluster variety called the cluster symplectic double. This cluster variety carries a natural symplectic form and plays a key role in the quantization of cluster Poisson varieties~\cite{dilog}.

It was shown in~\cite{double} and~\cite{A} that this object is related to certain moduli spaces of geometric structures on a doubled surface. Given a compact oriented surface $S$ with boundary, one considers the same surface~$S^\circ$ with the opposite orientation. The \emph{double} $S_{\mathcal{D}}$ of~$S$ is then defined as the surface obtained by gluing $S$ and $S^\circ$ along corresponding boundary components and deleting the image of each marked point on the boundary of~$S$ to get a punctured surface. The central object in~\cite{double} is a moduli space $\mathcal{D}_{PGL_m,S}$. It is closely related to the space of twisted $SL_m$-local systems on the double~$S_{\mathcal{D}}$, together with additional framing data. In fact, the correct definition of this space is much more subtle and involves quotienting by the action of a certain finite group. (For the precise definition, see Section~\ref{sec:ModuliSpaces} below.) In~\cite{double}, Fock and Goncharov describe a rational map 
\[
\mathcal{D}_{PGL_m,S}\dashrightarrow\mathcal{D}
\]
from the moduli space to the cluster symplectic double variety associated with a surface~$S$. The main result of the present paper is the following extension of this result of Fock and Goncharov:

\begin{theorem}
\label{thm:intromain}
The rational map above is a birational equivalence.
\end{theorem}

This result gives an analog for the cluster symplectic double of the results of~\cite{IHES}, which identify the cluster Poisson and $K_2$-varieties with the moduli spaces $\mathcal{X}_{PGL_m,S}$ and $\mathcal{A}_{SL_m,S}$, respectively.

The rest of this paper is organized as follows. In~Section~\ref{sec:ClusterVarieties}, we define the three types of cluster varieties following~\cite{dilog}. In~Section~\ref{sec:ModuliSpaces}, we recall the notions of framed local systems and decorated twisted local systems from~\cite{IHES}. We also recall the symplectic double moduli space from~\cite{double}. Finally, in~Section~\ref{sec:TheMainResult}, we define the above rational map and prove that it gives a birational equivalence.

In this paper, all schemes are defined over the field of rational numbers.

\section{Cluster varieties}
\label{sec:ClusterVarieties}

\subsection{Cluster $K_2$- and Poisson varieties}

We begin by reviewing the three types of cluster varieties defined in~\cite{dilog}.

\begin{definition}
A \emph{seed} $\mathbf{i}=(I,J,\varepsilon_{ij})$ consists of a finite set $I$, a subset $J\subseteq I$, and a skew-symmetric matrix $\varepsilon_{ij}$ ($i,j\in I$) with integer entries. The matrix $\varepsilon_{ij}$ is called the \emph{exchange matrix}, and the set $I-J$ is called the set of \emph{frozen elements} of $I$.
\end{definition}

Given a seed $\mathbf{i}=(I,J,\varepsilon_{ij})$, we get two split algebraic tori $\mathcal{X}_\mathbf{i} = (\mathbb{G}_m)^{|J|}$ and $\mathcal{A}_\mathbf{i} = (\mathbb{G}_m)^{|I|}$. Let $\{X_j\}$ be the natural coordinates on~$\mathcal{X}_\mathbf{i}$ and $\{A_i\}$ the natural coordinates on~$\mathcal{A}_\mathbf{i}$.

\begin{definition}
Let $\mathbf{i}=(I,J,\varepsilon_{ij})$ be a seed and $k\in J$ a non-frozen element. Then we define a new seed $\mu_k(\mathbf{i})=\mathbf{i}'=(I',J',\varepsilon_{ij}')$, called the seed obtained by \emph{mutation} in the direction~$k$ by setting $I'=I$, $J'=J$, and 
\[
\varepsilon_{ij}'=
\begin{cases}
-\varepsilon_{ij} & \mbox{if } k\in\{i,j\} \\
\varepsilon_{ij}+\frac{|\varepsilon_{ik}|\varepsilon_{kj}+\varepsilon_{ik}|\varepsilon_{kj}|}{2} & \mbox{if } k\not\in\{i,j\}.
\end{cases}
\]
\end{definition}

Two seeds will be called \emph{mutation equivalent} if they are related by a sequence of mutations. We will denote the mutation equivalence class of a seed $\mathbf{i}$ by $|\mathbf{i}|$.

A seed mutation induces birational maps on tori defined by the formulas 
\[
\mu_k^*A_i' =
\begin{cases}
A_k^{-1}\biggr(\prod_{j|\varepsilon_{kj>0}}A_j^{\varepsilon_{kj}} + \prod_{j|\varepsilon_{kj<0}}A_j^{-\varepsilon_{kj}}\biggr) & \mbox{if } i=k \\
A_i & \mbox{if } i\neq k
\end{cases}
\]
and
\[
\mu_k^*X_i'=
\begin{cases}
X_k^{-1} & \mbox{if } i=k \\
X_i{(1+X_k^{-\sgn(\varepsilon_{ik})})}^{-\varepsilon_{ik}} & \mbox{if } i\neq k
\end{cases}
\]
where $A_i'$ and $X_i'$ are the coordinates on $\mathcal{A}_{\mathbf{i}'}$ and $\mathcal{X}_{\mathbf{i}'}$. A  transformation of the $\mathcal{A}$- or $\mathcal{X}$-tori obtained by composing these birational maps is called a \emph{cluster transformation}.

\Needspace*{2\baselineskip}
\begin{definition} \mbox{}
\begin{enumerate}
\item The \emph{cluster $K_2$-variety} $\mathcal{A}=\mathcal{A}_{|\mathbf{i}|}$ is a scheme obtained by gluing the $\mathcal{A}$-tori for all seeds mutation equivalent to the seed $\mathbf{i}$ using the birational maps above.
\item The \emph{cluster Poisson variety} $\mathcal{X}=\mathcal{X}_{|\mathbf{i}|}$ is obtained by gluing the $\mathcal{X}$-tori for all seeds mutation equivalent to the seed $\mathbf{i}$ using the birational maps above. The scheme obtained in this way may be nonseparated.
\end{enumerate}
\end{definition}

There is a mapping $p:\mathcal{A}\rightarrow\mathcal{X}$ given in any cluster coordinate system by $p^*(X_i)=\prod_{j\in I}A_j^{\varepsilon_{ij}}$.

\subsection{The cluster symplectic double}
\label{sec:TheClusterSymplecticDouble}

For any seed $\mathbf{i}=(I,J,\varepsilon_{ij})$, we now consider a split algebraic torus $\mathcal{D}_\mathbf{i} = (\mathbb{G}_m)^{2|J|}$ with natural coordinates $\{B_i, X_i\}_{i\in J}$. A seed mutation induces a birational map on tori defined by 
\[
\mu_k^*X_i'=
\begin{cases}
X_k^{-1} & \mbox{if } i=k \\
X_i{(1+X_k^{-\sgn(\varepsilon_{ik})})}^{-\varepsilon_{ik}} & \mbox{if } i\neq k
\end{cases}
\]
and 
\[
\mu_k^*B_i' =
\begin{cases}
\frac{X_k\prod_{j|\varepsilon_{kj>0}}B_j^{\varepsilon_{kj}} + \prod_{j|\varepsilon_{kj<0}}B_j^{-\varepsilon_{kj}}}{(1+X_k)B_k} & \mbox{if } i=k \\
B_i & \mbox{if } i\neq k
\end{cases}
\]
where $X_i'$ and $B_i'$ are the coordinates on $\mathcal{D}_{\mathbf{i}'}$.

As before, a transformation of the $\mathcal{D}$-tori obtained by composing these birational maps is called a \emph{cluster transformation}.

\begin{definition}
The \emph{cluster symplectic double} $\mathcal{D}=\mathcal{D}_{|\mathbf{i}|}$ is a scheme obtained by gluing the $\mathcal{D}$-tori for all seeds mutation equivalent to the seed $\mathbf{i}$ using the above maps.
\end{definition}

There are various relationships between the different cluster varieties. In the following result, we will assume for simplicity that the seeds used to define $\mathcal{A}$, $\mathcal{X}$, and~$\mathcal{D}$ have no frozen elements.

\begin{theorem}[\cite{dilog}, Theorem~2.3]
\label{thm:doubleproperties}
Let $\mathbf{i}=(I,J,\varepsilon_{ij})$ be a seed with $I=J$. Then the cluster symplectic double $\mathcal{D}=\mathcal{D}_{|\mathbf{i}|}$ satisfies the following properties.
\begin{enumerate}
\item There is a map $\varphi:\mathcal{A}\times\mathcal{A}\rightarrow\mathcal{D}$ given in any cluster coordinate system by the formulas 
\begin{align*}
\varphi^*(X_i) &= \prod_j A_j^{\varepsilon_{ij}} \\
\varphi^*(B_i) &= \frac{A_i^\circ}{A_i}
\end{align*}
where $A_i^\circ$ are the coordinates on the second factor of $\mathcal{A}$.
\item There is a map $\pi:\mathcal{D}\rightarrow \mathcal{X}\times\mathcal{X}$ given in any cluster coordinate system by the formulas 
\begin{align*}
\pi^*(X_i\otimes1) &= X_i \\
\pi^*(1\otimes X_i) &= X_i \prod_{j\in J}B_j^{\varepsilon_{ij}}.
\end{align*}
\item There are commutative diagrams 
\[
\vcenter{
\xymatrix{ 
\mathcal{A}\times\mathcal{A} \ar[rd]^-{\varphi} \ar[dd]_{p\times p} \\
& \mathcal{D} \ar[ld]^-{\pi} \\
\mathcal{X}\times\mathcal{X}
}
}
\quad
\vcenter{
\xymatrix{ 
\mathcal{X} \ar[r]^{j} \ar[d] & \mathcal{D} \ar[d]^{\pi} \\
\Delta_\mathcal{X} \ar@{^(->}[r] & \mathcal{X}\times\mathcal{X}
}
}
\]
where $\Delta_\mathcal{X}$ denotes the diagonal in $\mathcal{X}\times\mathcal{X}$. Here $j:\mathcal{X}\rightarrow\mathcal{D}$ is an embedding whose image is a Lagrangian subspace given in any coordinate system by $B_i=1$~($i\in J$).
\item There is an involutive isomorphism $\iota:\mathcal{D}\rightarrow\mathcal{D}$ which interchanges the two components of the map $\pi$ and is given in any cluster coordinate system by the formulas 
\begin{align*}
\iota^*(B_i) &= B_i^{-1} \\
\iota^*(X_i) &= X_i \prod_{j\in J}B_j^{\varepsilon_{ij}}.
\end{align*}
\end{enumerate}
\end{theorem}

\begin{proof} 
See \cite{dilog}, Sections 2 and 3.
\end{proof}

\section{Moduli spaces}
\label{sec:ModuliSpaces}

\subsection{Preliminaries on surfaces}

In this section, we define the three moduli spaces that play a role in our discussion. We begin by introducing some terminology related to surfaces.

\begin{definition}
A \emph{decorated surface} is a compact oriented surface with boundary together with a finite (possibly empty) collection of marked points on the boundary.
\end{definition}

Given a decorated surface $S$, we can shrink those boundary components without marked points to get a surface $S'$ with punctures and boundary where every boundary component contains at least one marked point.

\begin{definition}
Let $S$ be a decorated surface. An \emph{ideal triangulation} $T$ of $S$ is a triangulation of the surface $S'$ described in the preceding paragraph whose vertices are the marked points and the punctures.
\end{definition}

From now on, we will consider only decorated surfaces $S$ that admit an ideal triangulation. Note that in general the sides of a triangle in an ideal triangulation may not be distinct. In this case, the triangle is said to be \emph{self-folded}.

An edge of an ideal triangulation $T$ is called \emph{external} if it lies along the boundary of $S$, connecting two marked points, and is called \emph{internal} otherwise. If $k$ is an internal edge of the ideal triangulation $T$, then a \emph{flip} at $k$ is the transformation of $T$ that removes the edge $k$ and replaces it by the unique different edge that, together with the remaining edges, forms a new ideal triangulation:
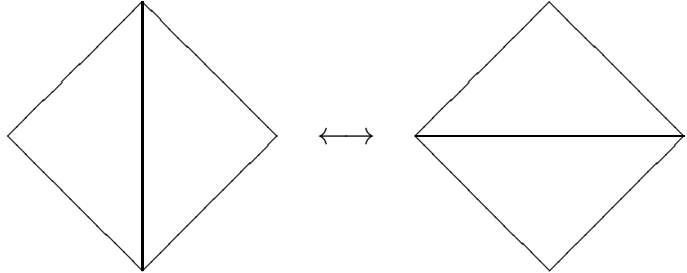
\[
\xy /l1.5pc/:
{\xypolygon4"A"{~:{(2,2):}}},
{"A1"\PATH~={**@{-}}'"A3"},
\endxy
\quad
\longleftrightarrow
\quad
\xy /l1.5pc/:
{\xypolygon4"A"{~:{(2,2):}}},
{"A2"\PATH~={**@{-}}'"A4"}
\endxy
\]
A flip will be called \emph{regular} if none of the triangles above is self-folded. It is a fact that any two isotopy classes of ideal triangulations on a surface are related by a sequence of flips.

Let $m\geq2$ be an integer and consider the triangle in $\mathbb{R}^3$ defined by the equation 
\[
x+y+z=m
\]
where $x,y,z\geq0$. We can subdivide this into smaller triangles by drawing the lines $x=p$, $y=p$, and~$z=p$ where $p$  is an integer, $0\leq p\leq m$. An \emph{$m$-triangulation} of the triangle is defined to be a triangulation isotopic to this one. Given an ideal triangulation $T$ of a surface, we can draw a homeomorphic image of an $m$-triangulation within each of its triangles, matching up the vertices on the edges of~$T$. This produces a new triangulation called the \emph{$m$-triangulation} of $T$. The illustration below shows a pair of triangles and the corresponding 3-triangulation.
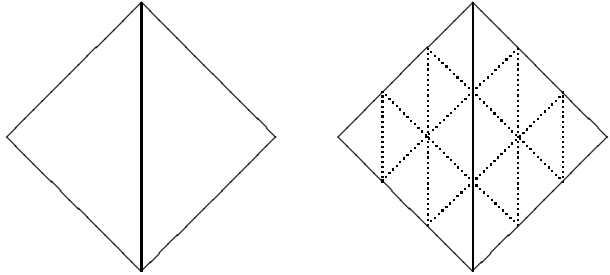
\[
\xy /l1.5pc/:
{\xypolygon4"A"{~:{(2,2):}}},
"A1";"A3" **\dir{-}; 
\endxy
\quad
\quad
\xy /l1.5pc/:
{\xypolygon4"A"{~:{(2,2):}}},
"A1";"A3" **\dir{-}; 
(1,-0.95)*{}="W1"; 
(1,0.95)*{}="W2"; 
(0.05,-1.85)*{}="X1"; 
(-0.9,-0.95)*{}="X2"; 
(0.05,1.85)*{}="Y1"; 
(-0.9,0.95)*{}="Y2"; 
(1.95,-1.85)*{}="U1"; 
(2.9,-0.95)*{}="U2"; 
(1.95,1.85)*{}="V1"; 
(2.9,0.95)*{}="V2"; 
"X1";"Y1" **\dir{.}; 
"X2";"Y2" **\dir{.}; 
"X1";"W1" **\dir{.}; 
"X2";"W2" **\dir{.}; 
"Y2";"W1" **\dir{.}; 
"Y1";"W2" **\dir{.}; 
"U1";"V1" **\dir{.}; 
"U2";"V2" **\dir{.}; 
"U1";"W1" **\dir{.}; 
"U2";"W2" **\dir{.}; 
"V2";"W1" **\dir{.}; 
"V1";"W2" **\dir{.}; 
\endxy
\]

The orientation of the surface $S$ provides an orientation of any triangle in the ideal triangulation~$T$. If $e$ is any edge of the $m$-triangulation that does not lie along an edge of the original ideal triangulation~$T$, then $e$ is parallel to one of the edges of $T$, and therefore it acquires an orientation.

Given an ideal triangulation $T$ of a decorated surface $S$ and an integer $m\geq2$, we define 
\[
I_m^T=\{\text{vertices of the $m$-triangulation of $T$}\} - \{\text{vertices of $T$}\}
\]
and 
\[
J_m^T=I_m^T-\{\text{vertices on the boundary of $S$}\}.
\]

\begin{definition}
Let $T$ be an ideal triangulation of~$S$ with no self-folded triangles. The \emph{exchange matrix} $\varepsilon_{ij}^T$ ($i,j\in I_m^T\times I_m^T$) is given by the formula 
\[
\varepsilon_{ij}^T=|\{\text{oriented edges from $i$ to $j$}\}| - |\{\text{oriented edges from $j$ to $i$}\}|.
\]
\end{definition}

Thus for any decorated surface with an ideal triangulation~$T$, we can define a seed with $I=I_m^T$, $J=J_m^T$, and exchange matrix $\varepsilon_{ij}=\varepsilon_{ij}^T$.

\subsection{The classical moduli spaces}

\subsubsection{Framed local systems}

Recall that the \emph{flag variety} for the group $PGL_m$ is defined as the quotient space $\mathcal{B}=PGL_m/B$ where $B$ is a Borel subgroup of $PGL_m$. Let $\mathcal{L}$ be a $PGL_m$-local system, that is, a principal $PGL_m$-bundle with flat connection. There is a natural left action of the group $PGL_m$ on the flag variety $\mathcal{B}$, and so we can form the associated bundle 
\[
\mathcal{L}_{\mathcal{B}}=\mathcal{L}\times_{PGL_m}\mathcal{B}.
\]

Let $S$ be a decorated surface. Choose points $x_1,\dots,x_r$ on its boundary, one on each segment bounded by adjacent marked points. Then the \emph{punctured boundary} of $S$ is defined as $\partial^\circ S=\partial S-\{x_1,\dots,x_r\}$. We will consider local systems on $S$ with some additional data.

\begin{definition}
Let $S$ be a decorated surface and $\mathcal{L}$ a $PGL_m$-local system on $S$. A \emph{framing} for $\mathcal{L}$ is a flat section of the restriction $\mathcal{L}_{\mathcal{B}}|_{\partial^\circ S}$ to the punctured boundary. A \emph{framed $PGL_m$-local system} is a local system $\mathcal{L}$ together with a framing. The space of all framed $PGL_m$-local systems on $S$ is denoted~$\mathcal{X}_{PGL_m,S}$.
\end{definition}

There is an alternative way of thinking about framed local systems that will be useful in what follows. Let $S$ be a decorated surface as before, and let $S'$ be the punctured surface obtained by deleting all marked points and shrinking those boundary components without marked points. Fix a complete, finite-area hyperbolic metric on this surface~$S'$ so that $\partial S'$ is totally geodesic. Then its universal cover can be identified with a subset of the hyperbolic plane~$\mathbb{H}$ with totally geodesic boundary. The punctures and deleted marked points in~$S'$ give rise to a set $\mathcal{F}_{\infty}(S)$ of points on the boundary~$\partial\mathbb{H}$. The action of $\pi_1(S)$ by deck transformations on the universal cover gives rise to an action of $\pi_1(S)$ on this set~$\mathcal{F}_{\infty}(S)$.

Recall that a $PGL_m(\mathbb{C})$-local system can be viewed as a homomorphism $\rho:\pi_1(S)\rightarrow PGL_m(\mathbb{C})$, modulo the action of $PGL_m(\mathbb{C})$ by conjugation. The following result gives a characterization of framed local systems emphasizing this monodromy representation.

\begin{proposition}[\cite{IHES}, Lemma~1.1]
\label{prop:Xconfig}
Consider a pair $(\rho,\psi)$ where $\rho:\pi_1(S)\rightarrow PGL_m(\mathbb{C})$ is a group homomorphism, and $\psi:\mathcal{F}_{\infty}(S)\rightarrow\mathcal{B}(\mathbb{C})$ is a $(\pi_1(S),\rho)$-equivariant map from the set described above into the flag variety. That is, for any $\gamma\in\pi_1(S)$, we have 
\[
\psi(\gamma c)=\rho(\gamma)\psi(c).
\]
The moduli space $\mathcal{X}_{PGL_m,S}(\mathbb{C})$ parametrizes such pairs modulo the action of the group $PGL_m(\mathbb{C})$.
\end{proposition}

In~\cite{IHES}, Fock and Goncharov show how to associate, to a general point of $\mathcal{X}_{PGL_m,S}$ and ideal triangulation $T$ of $S$, a collection of coordinates $X_j$ indexed by the set $J=J_m^T$ described above. These coordinates determine a point of the torus $\mathcal{X}_{\mathbf{i}}$ where $\mathbf{i}$ is the seed corresponding to the triangulation~$T$. In fact, Fock and Goncharov prove the following result:

\begin{theorem}[\cite{IHES}, Theorem~1.17]
\label{thm:Xisom}
There is a birational map 
\[
\mathcal{X}_{PGL_m,S}\dashrightarrow\mathcal{X}_{\mathbf{i}}
\]
where $\mathbf{i}$ is the seed corresponding to an ideal triangulation $T$ of~$S$. Let $T$ and $T'$ be ideal triangulations of~$S$ related by a sequence of regular flips, and let $\mathbf{i}$ and~$\mathbf{i}'$ be the corresponding seeds. Then the transition map $\mathcal{X}_{\mathbf{i}}\dashrightarrow\mathcal{X}_{\mathbf{i}'}$ is the cluster transformation used to glue these tori in the cluster Poisson variety.
\end{theorem}

\subsubsection{Decorated twisted local systems}

In addition to the ordinary flag variety, we consider the \emph{decorated flag variety} of $SL_m$, which is isomorphic to the quotient $\mathcal{A}=SL_m/U$ where $U$ is a maximal unipotent subgroup of~$SL_m$. Let $\mathcal{L}$ be an $SL_m$-local system. There is a natural left action of~$SL_m$ on $\mathcal{A}$, so we can form the associated bundle 
\[
\mathcal{L}_{\mathcal{A}}=\mathcal{L}\times_{SL_m}\mathcal{A}.
\]

For a given surface $S$ as above, we can also consider local systems on the punctured tangent bundle $T'S$, that is, the tangent bundle of~$S$ with the zero section removed. For any point $y\in S$, we have $T_yS\cong\mathbb{R}^2$. Thus $T_y'S=T_yS-0$ is homotopy equivalent to a circle, and we have 
\[
\pi_1(T_y'S,x)\cong\mathbb{Z}
\]
for any choice of basepoint $x\in T_y'S$. Let $\sigma_S$ denote a generator of this fundamental group. It is well defined up to a sign. By abuse of notation, we will also write $\sigma_S$ for the image of this generator under the inclusion $\pi_1(T_y'S,x)\hookrightarrow \pi_1(T'S,x)$. The group $\pi_1(T'S,x)$ fits into a short exact sequence 
\[
\xymatrix{ 
1 \ar[r] & \mathbb{Z} \ar[r] & \pi_1(T'S,x) \ar[r] & \pi_1(S,y) \ar[r] & 1
}
\]
where the group $\mathbb{Z}$ is identified with the central subgroup of $\pi_1(T'S,x)$ generated by~$\sigma_S$.

\begin{definition}
A \emph{twisted $SL_m$-local system} $\mathcal{L}$ on $S$ is an $SL_m$-local system on the punctured tangent bundle $T'S$ with monodromy $(-1)^{m-1}e$ around $\sigma_S$ where $e$ is the identity in~$SL_m$. A \emph{decoration} for $\mathcal{L}$ is a locally constant section of the restriction of $\mathcal{L}_{\mathcal{A}}$ to the punctured boundary. A \emph{decorated twisted $SL_m$-local system} is a twisted local system $\mathcal{L}$ together together with a decoration. The space of all decorated twisted $SL_m$-local systems on $S$ is denoted~$\mathcal{A}_{SL_m,S}$.
\end{definition}

Observe that the element $(-1)^{m-1}e$ has order two, and therefore this definition does not depend on the choice of generator $\sigma_S$. The existence of a decoration implies that a decorated twisted $SL_m$-local system has unipotent monodromy around any puncture.

A twisted $SL_m$-local system on $S$ is thus the same thing as a homomorphism $\rho:\pi_1(T'S,x)\rightarrow SL_m$, considered up to conjugation, such that $\rho(\sigma_S)=(-1)^{m-1}e$. If we define $\bar{\pi}_1(T'S,x)$ to be the quotient of $\pi_1(T'S,x)$ by the central subgroup $2\mathbb{Z}$ generated by the element~$\sigma_S^2$, then we obtain a new central extension 
\[
\xymatrix{ 
1 \ar[r] & \mathbb{Z}/2\mathbb{Z} \ar[r] & \bar{\pi}_1(T'S,x) \ar[r] & \pi_1(S,y) \ar[r] & 1.
}
\]
Let $\bar{\sigma}_S$ be the order two element in the quotient group $\mathbb{Z}/2\mathbb{Z}$. Then we can think of a twisted local system as a representation $\rho:\bar{\pi}_1(T'S,x)\rightarrow SL_m$, considered modulo conjugation, with the property that $\rho(\bar{\sigma}_S)=(-1)^{m-1}e$.

In~\cite{IHES}, Fock and Goncharov show how to associate, to a general point of $\mathcal{A}_{SL_m,S}$ and ideal triangulation $T$ of $S$, a collection of coordinates $A_i$ indexed by the set $I=I_m^T$. These coordinates determine a point of the torus $\mathcal{A}_{\mathbf{i}}$ where $\mathbf{i}$ is the seed corresponding to the triangulation~$T$. In fact, Fock and Goncharov prove the following result:

\begin{theorem}[\cite{IHES}, Theorem~1.17]
\label{thm:Aisom}
There is a birational map 
\[
\mathcal{A}_{SL_m,S}\dashrightarrow\mathcal{A}_{\mathbf{i}}
\]
where $\mathbf{i}$ is the seed corresponding to an ideal triangulation $T$ of~$S$. Let $T$ and $T'$ be ideal triangulations of~$S$ related by a sequence of regular flips, and let $\mathbf{i}$ and~$\mathbf{i}'$ be the corresponding seeds. Then the transition map $\mathcal{A}_{\mathbf{i}}\dashrightarrow\mathcal{A}_{\mathbf{i}'}$ is the cluster transformation used to glue these tori in the cluster $K_2$-variety.
\end{theorem}

\subsection{The symplectic double moduli space}

Suppose that $S$ is a decorated surface and $S^\circ$ is the same surface equipped with the opposite orientation. The \emph{double} $S_{\mathcal{D}}$ is obtained by gluing $S$ and $S^\circ$ along corresponding boundary components and deleting the image of each marked point on the boundary of~$S$ to get a punctured surface.

Denote by $\widetilde{\Loc}_{SL_m,S}$ the space of twisted $SL_m$-local systems on the surface $S$. Suppose we are given a representation $\rho:\bar{\pi}_1(T'S)\rightarrow SL_m$ corresponding to a point in this space and a homomorphism $\sigma:\pi_1(S)\rightarrow Z(SL_m)$ into the center of~$SL_m$. Composing the latter map with the projection $\bar{\pi}_1(T'S)\rightarrow\pi_1(S)$, we get a homomorphism $\tilde{\sigma}:\bar{\pi}_1(T'S)\rightarrow SL_m$, which we can multiply pointwise with $\rho$ to get a new point $\sigma\cdot\rho\in\widetilde{\Loc}_{SL_m,S}$. Thus we have defined an action of the group 
\[
\Delta_{SL_m}=\Hom(\pi_1(S),Z(SL_m))
\]
on the space $\widetilde{\Loc}_{SL_m,S}$. The natural homeomorphism $S\rightarrow S^\circ$ induces an isomorphism $\pi_1(S)\cong\pi_1(S^\circ)$, and therefore we have a diagonal action of $\Delta_{SL_m}$ on the product space 
\[
\widetilde{\Loc}_{SL_m,S}\times\widetilde{\Loc}_{SL_m,S^\circ}.
\]

Let $\mathcal{L}$ be a twisted $SL_m$-local system on a decorated surface $S$. If $\beta$ is a framing for $\mathcal{L}$, then we define an $H$-local subsystem $\mathcal{F}_{\beta}$ over the punctured boundary of~$S$, where $H$ is the Cartan group for~$SL_m$. It is the local subsystem of $\mathcal{L}_{\mathcal{A}}$ obtained by taking the preimage of the section $\beta$ under the natural map $\mathcal{L}_{\mathcal{A}}\rightarrow\mathcal{L}_{\mathcal{B}}$.

We now give the definition of the symplectic double moduli space. As in~\cite{double}, we will assume that the surface~$S$ has no marked points on its boundary. Later we will explain how this construction can be extended to general decorated surfaces.

\begin{definition}[\cite{double}, Definition~2.3]
\label{def:Dspace}
Let $S$ be a decorated surface with no marked points. The moduli space $\mathcal{D}_{PGL_m,S}$ parametrizes the data $(\mathcal{L},\mathcal{L}^\circ,\beta,\beta^\circ,\alpha)$ where 
\begin{enumerate}
\item The pair $(\mathcal{L},\mathcal{L}^\circ)$ is an element of 
\[
\left(\widetilde{\Loc}_{SL_m,S}\times\widetilde{\Loc}_{SL_m,S^\circ}\right)/\Delta_{SL_m}.
\]

\item $\beta$ and $\beta^\circ$ are framings for the $PGL_m$-local systems on $S$ and $S^\circ$ corresponding to $\mathcal{L}$ and $\mathcal{L}^\circ$, respectively.

\item $\alpha$ is an $H$-equivariant map $\mathcal{F}_{\beta}\rightarrow\mathcal{F}_{\beta^\circ}$ of local subsystems. (So in particular these local subsystems are isomorphic.)
\end{enumerate}
\end{definition}

In~\cite{double}, Fock and Goncharov showed that this space is closely related to the geometry of the doubled surface $S_\mathcal{D}$. In~\cite{A}, the author studied versions of the Teichm\"uller space and space of measured laminations on~$S_\mathcal{D}$ which are closely related to the above construction.

Definition~\ref{def:Dspace} describes the moduli space $\mathcal{D}_{PGL_m,S}$ when there are no marked points on the boundary of~$S$. To define this space for a general decorated surface, we must modify the definition slightly. As before, we consider tuples $(\mathcal{L},\mathcal{L}^\circ,\beta,\beta^\circ,\alpha)$. If we choose a decorated flag in the fiber of $\mathcal{F}_\beta$ over each marked point on $\partial S$, then the map $\alpha$ gives a corresponding choice of decorated flags in the fibers of $\mathcal{F}_{\beta^\circ}$. For a generic choice of flags, Fock and Goncharov's construction provides coordinates $A_i$ and~$A_i^\circ$~($i\in I_m^T-J_m^T$) corresponding to these decorations of~$\mathcal{L}$ and~$\mathcal{L}^\circ$, respectively. They are independent of the ideal triangulation~$T$.

\begin{definition}
\label{def:Dspacemarked}
For any decorated surface $S$, the space $\mathcal{D}_{PGL_m,S}$ parametrizes the data $(\mathcal{L},\mathcal{L}^\circ,\beta,\beta^\circ,\alpha)$ as above where we require $A_i=A_i^\circ$ for $i\in I_m^T-J_m^T$ and any choice of decorated flags.
\end{definition}

Below we will see that this definition is a natural one as it allows us to prove our main result for arbitrary decorated surfaces.

\section{The main result}
\label{sec:TheMainResult}

\subsection{Construction of coordinates}

To construct coordinates on the moduli space $\mathcal{D}_{PGL_m,S}$ of the previous section, fix an ideal triangulation~$T$ of the surface~$S$ and a general point $\mu\in\mathcal{D}_{PGL_m,S}$. This point $\mu$ determines a framed local system on~$S$. By Theorem~\ref{thm:Xisom}, there is a collection of $X_j$, indexed by the set $J=J_m^T$, which are coordinates of this framed local system.

In addition to these coordinates $X_j$ ($j\in J$), we will define a collection of coordinates $B_j$~($j\in J$) as follows. Let $t$ be any triangle in the ideal triangulation~$T$. Then the data defining the point $m\in\mathcal{D}_{PGL_m,S}$ allow us to assign, to each vertex~$p$ of this triangle, an invariant flag $b_p$. For each vertex $p$, we can then choose a decorated flag~$a_p$ in the fiber of the projection $\mathcal{A}\rightarrow\mathcal{B}$ over the flag $b_p$. By parallel transporting these decorated flags to a common point in the interior of $t$, we get a well defined point in the configuration space 
\[
\Conf_3(\mathcal{A})=SL_m\backslash\mathcal{A}^3
\]
of triples of decorated flags. On the other hand, consider the ideal triangulation~$T^\circ$ of~$S^\circ$ corresponding to $T$, and let $t^\circ$ be the triangle in this triangulation corresponding to $t$. We have associated a decorated flag $a_p$ to each vertex $p$, and the choice of $\mu$ allows us to associate a decorated flag $a_p^\circ$ to the corresponding vertex of $t^\circ$. We can once again transport these decorated flags to a common point in the interior of $t^\circ$. In this way, we obtain a second point of $\Conf_3(\mathcal{A})$. This construction produces a well defined point in 
\[
\left(\Conf_3(\mathcal{A})\times\Conf_3(\mathcal{A})\right)/H^3
\]
where $H$ denotes the Cartan group of $SL_m$.

By the results of~\cite{IHES}, the first factor $\Conf_3(\mathcal{A})$ is identified with the space of decorated twisted local systems on $t$, and by Theorem~\ref{thm:Aisom}, there is a set of coordinates $A_i$ on this space, parametrized by a set of vertices of the $m$-triangulation of~$t$. Similarly, there are coordinates~$A_i^\circ$ on the second factor. If the surface $S$ has marked points, then we have $A_i=A_i^\circ$ for $i\in I_m^T-J_m^T$ by Definition~\ref{def:Dspacemarked}. For each index $j\in J_m^T$, we define 
\[
B_j=\frac{A_j^\circ}{A_j}.
\]
One can show that this ratio is independent of all choices in the construction. Thus we have associated a well defined $B_j$ to each index $j$ in the set $J_m^T$.

\begin{proposition}[\cite{double}]
This construction defines a rational map 
\[
\mathcal{D}_{PGL_m,S}\dashrightarrow\mathcal{D}_{\mathbf{i}}
\]
where $\mathbf{i}$ is the seed corresponding to an ideal triangulation $T$ of~$S$ and $\mathcal{D}_{\mathbf{i}}$ is the torus from Section~\ref{sec:TheClusterSymplecticDouble}. Let $T$ and $T'$ be ideal triangulations of~$S$ related by a sequence of regular flips, and let $\mathbf{i}$ and~$\mathbf{i}'$ be the corresponding seeds. Then the transition map $\mathcal{D}_{\mathbf{i}}\dashrightarrow\mathcal{D}_{\mathbf{i}'}$ is the cluster transformation used to glue these tori in the cluster symplectic double.
\end{proposition}

\subsection{Proof of the main theorem}

To prove our main result, we will need the following elementary facts.

\begin{lemma}
\label{lem:splitting}
Let 
\[
\xymatrix{ 
1 \ar[r] & A \ar[r]^{u} & B \ar[r]^{v} & C \ar[r] & 1
}
\]
be a central extension of groups with $\Aut(A)=1$, and let $s:C\rightarrow B$ be a homomorphism such that $v\circ s=1_C$. Then there is a natural isomorphism $A\times C\cong B$, $(a,c)\mapsto u(a)s(c)$. If $s':C\rightarrow B$ is any other homomorphism such that $v\circ s'=1_C$, then there exists $\varphi\in\Hom(C,A)$ such that 
\[
s'(c)=u(\varphi(c))s(c)
\]
for all $c\in C$.
\end{lemma}

\begin{proof}
It is well known that if $s:C\rightarrow B$ is a homomorphism satisfying $v\circ s=1_C$, then the rule $(a,c)\mapsto u(a)s(c)$ defines an isomorphism of a semidirect product $A\rtimes C$ with~$B$. Since $\Aut(A)=1$, the semidirect product is isomorphic to the direct product $A\times C$.

If $s$ and $s'$ are homomorphisms $C\rightarrow B$ such that $v\circ s=v\circ s'=1_C$, then we have 
\[
v(s'(c)s(c)^{-1})=1
\]
for all $c\in C$. By exactness, we have $s'(c)s(c)^{-1}\in\ker v=\im u$, and therefore there exists $\varphi(c)\in A$ such that $u(\varphi(c))=s'(c)s(c)^{-1}$, that is,
\[
s'(c)=u(\varphi(c))s(c)
\]
for all $c\in C$. We claim that the map $\varphi:C\rightarrow A$ defined in this way is a group homomorphism. Indeed, suppose $c_1$,~$c_2\in C$. Since the image of $u$ is central in $B$, we have 
\begin{align*}
u(\varphi(c_1c_2)) &= s'(c_1c_2)s(c_1c_2)^{-1} \\
&= s'(c_1)s'(c_2)s(c_2)^{-1}s(c_1)^{-1} \\
&= s'(c_1)u(\varphi(c_2))s(c_1)^{-1} \\
&= s'(c_1)s(c_1)^{-1}u(\varphi(c_2)) \\
&= u(\varphi(c_1))u(\varphi(c_2)).
\end{align*}
Since $u$ is injective, we can cancel $u$ on both sides of $u(\varphi(c_1c_2))=u(\varphi(c_1)\varphi(c_2))$ to see that $\varphi$ is a homomorphism as claimed. This completes the proof.
\end{proof}

\begin{lemma}
\label{lem:lifting}
Let $S$ be a decorated surface with at least one boundary component. Any homomorphism $\pi_1(S)\rightarrow PGL_m(\mathbb{C})$ can be lifted to a homomorphism $\pi_1(S)\rightarrow SL_m(\mathbb{C})$. Any two lifts are related by the action of an element of the group 
\[
\Delta_{SL_m}=\Hom(\pi_1(S),Z(SL_m(\mathbb{C}))).
\]
\end{lemma}

\begin{proof}
Let $\rho:\pi_1(S)\rightarrow PGL_m(\mathbb{C})$ be a homomorphism. Since $\pi_1(S)$ is free and $\mathbb{C}$ is algebraically closed, we can lift this to a homomorphism $\pi_1(S)\rightarrow SL_m(\mathbb{C})$ by choosing a lift for each generator of $\pi_1(S)$.

Suppose $\rho_1$,~$\rho_2:\pi_1(S)\rightarrow SL_m(\mathbb{C})$ are lifts of $\rho$. Then for every $\gamma\in\pi_1(S)$, the elements $\rho_1(\gamma)$ and $\rho_2(\gamma)$ represent the same class in $PGL_m(\mathbb{C})$, so there exists a nonzero scalar $\lambda_{\gamma}$ such that 
\[
\rho_2(\gamma)=\lambda_{\gamma}\rho_1(\gamma).
\]
Taking determinants of both sides, we see that 
\[
1=\lambda_{\gamma}^m.
\]
The group $Z(SL_m(\mathbb{C}))$ is identified with the group of $m$th roots of unity, so we can define a map $\sigma:\pi_1(S)\rightarrow Z(SL_m(\mathbb{C}))$ by putting $\sigma(\gamma)=\lambda_{\gamma}$. One easily checks that this map is a homomorphism. By construction, we have $\rho_2=\sigma\cdot\rho_1$.
\end{proof}

Let $A=\mathbb{Q}[x_1,\dots,x_k]/(f_1,\dots,f_s)$ and $B=\mathbb{Q}[x_1,\dots,x_l]/(g_1,\dots,g_t)$ be reduced algebras over~$\mathbb{Q}$, and let $X=\Spec A$ and $Y=\Spec B$ be the corresponding affine schemes over~$\mathbb{Q}$.

\begin{lemma}
\label{lem:pointsfield}
Let $\varphi^*:B\rightarrow A$ and $\psi^*:A\rightarrow B$ be ring homomorphisms such that the induced maps $\varphi:X(\mathbb{C})\rightarrow Y(\mathbb{C})$ and $\psi:Y(\mathbb{C})\rightarrow X(\mathbb{C})$ satisfy $\varphi\circ\psi=1_{Y(\mathbb{C})}$ and $\psi\circ\varphi=1_{X(\mathbb{C})}$. Then $X$ and $Y$ are isomorphic as schemes over~$\mathbb{Q}$.
\end{lemma}

\begin{proof}
We can think of the $X(\mathbb{C})$ as the closed subset of $\mathbb{C}^k$ defined by the polynomials $f_1,\dots,f_s\in\mathbb{C}[x_1,\dots,x_k]$ and think of $Y(\mathbb{C})$ as the closed subset of $\mathbb{C}^l$ defined by $g_1,\dots,g_t\in\mathbb{C}[x_1,\dots,x_l]$. The maps $\psi$ and $\varphi$ provide inverse isomorphisms of these closed subsets. Since we assume that $A$ and $B$ are reduced, the Nullstellensatz implies that there is an isomorphism of coordinate rings  
\[
\mathbb{C}[x_1,\dots,x_k]/(f_1,\dots,f_s)\stackrel{\cong}{\rightarrow}\mathbb{C}[x_1,\dots,x_l]/(g_1,\dots,g_t).
\]
This isomorphism restricts to the map $\psi^*:A\rightarrow B$, and its inverse restricts to $\varphi^*:B\rightarrow A$. It follows that $A\cong B$ as rings, and therefore $X$ and $Y$ are isomorphic as schemes over~$\mathbb{Q}$.
\end{proof}

The following theorem is the main result of this section.

\begin{theorem}
\label{thm:main}
Fix an ideal triangulation $T$ of the surface $S$, and let $\mathbf{i}$ be the corresponding seed. There is a rational map 
\[
\mathcal{D}_{\mathbf{i}}\dashrightarrow\mathcal{D}_{PGL_m,S}
\]
from the split algebraic torus $\mathcal{D}_{\mathbf{i}}=\mathbb{G}_m^{2|J|}$ into the symplectic double moduli space. Composing this map with the coordinate functions gives the identity whenever the composition is defined.
\end{theorem}

\begin{proof}
The proof will consist of two steps. In the first step, we will construct an inverse to this map at the level of complex points. The definition of this map will require some choices. In the second part of the proof, we will show that the construction is independent of these choices.

\Needspace*{2\baselineskip}
\begin{step}[1]
Construction of the map
\end{step}

To construct this map, suppose we are given $B_j$,~$X_j\in \mathbb{C}^*$ for~$j\in J$. By Theorem~\ref{thm:Xisom}, we can use the $X_j$ ($j\in J$) to construct a point in $\mathcal{X}_{PGL_m,S}(\mathbb{C})$, that is, a $PGL_m(\mathbb{C})$-local system on~$S$ together with a framing~$\beta$. This local system is represented by some homomorphism $\pi_1(S)\rightarrow PGL_m(\mathbb{C})$.

By Lemma~\ref{lem:lifting}, this homomorphism can be lifted to a homomorphism $\rho_0:\pi_1(S)\rightarrow SL_m(\mathbb{C})$. Recall the central extension 
\[
\xymatrix{ 
1 \ar[r] & \mathbb{Z}/2\mathbb{Z} \ar[r]^-{u} & \bar{\pi}_1(T'S) \ar[r]^-{v} & \pi_1(S) \ar[r] & 1.
}
\]
Since $\pi_1(S)$ is free, we can choose a homomorphism $s:\pi_1(S)\rightarrow\bar{\pi}_1(T'S)$ such that $v\circ s=1$. Once we have chosen this homomorphism, Lemma~\ref{lem:splitting} gives an isomorphism 
\[
\eta:\mathbb{Z}/2\mathbb{Z}\times\pi_1(S)\stackrel{\cong}{\rightarrow}\bar{\pi}_1(T'S).
\]
Consider the homomorphism $\mathbb{Z}/2\mathbb{Z}\times\pi_1(S)\rightarrow SL_m(\mathbb{C})$ defined by the rules $(1,\gamma)\mapsto\rho_0(\gamma)$ and $(\bar{\sigma}_S,1)\mapsto(-1)^{m-1}e$. Abusing notation, we will denote this map also by~$\rho_0$. Then the composition $\rho=\rho_0\circ\eta^{-1}:\bar{\pi}_1(T'S)\rightarrow SL_m(\mathbb{C})$ sends $\bar{\sigma}_S$ to the element $(-1)^{m-1}e$ and hence defines a twisted local system $\mathcal{L}$ on the surface~$S$.

Choose a complete, finite area hyperbolic metric with totally geodesic boundary on~$S'$ so that its universal cover is identified with a subset of the hyperbolic plane. By Proposition~\ref{prop:Xconfig}, the framing $\beta$ is the same thing as an equivariant map $\psi:\mathcal{F}_{\infty}(S)\rightarrow\mathcal{B}(\mathbb{C})$ that associates a flag to each point of $\mathcal{F}_{\infty}(S)$. Here $\mathcal{F}_{\infty}(S)$ is defined as the set of vertices of a lift $\tilde{T}$ of the ideal triangulation~$T$ to the hyperbolic plane.

Let $t$ be any triangle of this lifted triangulation. Then there is a flag associated to each vertex of $t$, and we can choose a decorated flag that projects to this flag under the natural map $\mathcal{A}(\mathbb{C})\rightarrow\mathcal{B}(\mathbb{C})$. This gives a triple of decorated flags and hence an element of $\mathcal{A}_{SL_m,t}(\mathbb{C})$. By Theorem~\ref{thm:Aisom}, there are coordinates $A_i$ on this space, corresponding to vertices $i$ in the $m$-triangulation of $t$. Define a new collection of coordinates by putting 
\[
A_i^\circ=B_iA_i
\]
for each vertex $i$. Here we take $B_i=1$ for $i\in I-J$.

Consider a second copy of the hyperbolic plane with a triangulation $\tilde{T}^\circ$ obtained by reversing the orientation of $\tilde{T}$. There is a triangle $t^\circ$ in this triangulation corresponding to the triangle~$t$, and we can use the numbers $A_i^\circ$ to construct a configuration of three decorated flags associated to the vertices of this triangle. In this way, we get an $H$-equivariant correspondence between decorated flags associated to the vertices of~$\tilde{T}$ and decorated flags associated to the vertices of~$\tilde{T}^\circ$.

Let $\gamma^\circ\in\pi_1(S^\circ)$ and let $\gamma$ be the corresponding element of $\pi_1(S)$. We can view this element $\gamma$ as a deck transformation of the universal cover of $S$, and it maps any triangle~$t$ of~$\tilde{T}$ to a new triangle $t'$. If we choose a decorated flag at each vertex of~$t$ which projects to the ordinary flag at this vertex, then the $SL_m(\mathbb{C})$ transformation $\rho_0(\gamma)$ gives a new triple of decorated flags at the vertices of~$t'$. Consider the corresponding triangles~$t^\circ$ and~$(t')^\circ$ in the triangulation~$\tilde{T^\circ}$. By the construction described above, there are corresponding decorated flags at the vertices of~$t^\circ$ and~$(t')^\circ$. We define $\rho_0^\circ(\gamma^\circ)$ to be the unique element of~$SL_m(\mathbb{C})$ that takes the triple of decorated flags at the vertices of~$t^\circ$ to the triple at the vertices of~$(t')^\circ$.

This defines a homomorphism $\rho_0^\circ:\pi_1(S^\circ)\rightarrow SL_m(\mathbb{C})$. From the isomorphism $\eta$ considered above, we get an isomorphism 
\[
\bar{\pi}_1(T'S^\circ)\cong\mathbb{Z}/2\mathbb{Z}\times\pi_1(S^\circ).
\]
Thus we can construct a representation $\rho^\circ:\bar{\pi}_1(T'S^\circ)\rightarrow SL_m(\mathbb{C})$ as before, and this defines a twisted $SL_m(\mathbb{C})$-local system $\mathcal{L}^\circ$ on the surface~$S^\circ$. The construction also gives an equivariant assignment of a flag to each vertex of $\tilde{T}^\circ$, so by Proposition~\ref{prop:Xconfig}, we have a framing of the associated $PGL_m(\mathbb{C})$-local system. Finally, the correspondence between decorated flags at the vertices of~$\tilde{T}$ and~$\tilde{T}^\circ$ gives the gluing datum~$\alpha$ appearing in~Definition~\ref{def:Dspace}.

Thus we have associated a point in $\mathcal{D}_{PGL_m,S}(\mathbb{C})$ to a collection of elements $B_j$,~$X_j\in \mathbb{C}^*$ for~$j\in J$. By construction, this map composes with the coordinate functions to give the identity whenever this composition is defined. This completes Step~1 of the proof.

\Needspace*{2\baselineskip}
\begin{step}[2]
Independence of choices
\end{step}

In the above construction, we had to choose a lift $\rho_0:\pi_1(S)\rightarrow SL_m(\mathbb{C})$ of a certain map $\pi_1(S)\rightarrow PGL_m(\mathbb{C})$. It follows from Lemma~\ref{lem:lifting} that the resulting pair $(\rho_0,\rho_0^\circ)$ of representations $\pi_1(S)\rightarrow SL_m(\mathbb{C})$ is well defined modulo the diagonal action of $\Delta_{SL_m}$.

In addition to this choice of lift, we had to choose a homomorphism $s:\pi_1(S)\rightarrow\bar{\pi}_1(T'S)$ with the property that $v\circ s=1$. By~Lemma~\ref{lem:splitting}, this choice provides an isomorphism $\eta:\mathbb{Z}/2\mathbb{Z}\times\pi_1(S)\rightarrow\bar{\pi}_1(T'S)$ given by the formula 
\[
\eta(x,\gamma)=u(x)s(\gamma).
\]
If we choose a different homomorphism $s':\pi_1(S)\rightarrow\bar{\pi}_1(T'S)$ satisfying $v\circ s'=1$, then we get a different isomorphism $\eta':\mathbb{Z}/2\mathbb{Z}\times\pi_1(S)\rightarrow\bar{\pi}_1(T'S)$ given by 
\[
\eta'(x,\gamma) = u(x)s'(\gamma) = u(\varphi(\gamma))u(x)s(\gamma)
\]
for some $\varphi\in\Hom(\pi_1(S),\mathbb{Z}/2\mathbb{Z})$. We constructed the twisted local system $\mathcal{L}$ above by describing a homomorphism $\rho:\bar{\pi}_1(T'S)\rightarrow SL_m(\mathbb{C})$. This was defined as a composition $\rho=\rho_0\circ\eta^{-1}$. If we repeat the same construction with the different map~$s'$, we get a different homomorphism $\rho'=\rho_0\circ\eta'^{-1}$. We can write 
\[
\rho=\rho_0\circ\eta^{-1}\circ\eta'\circ\eta'^{-1}.
\]
By the above formulas, we have 
\[
\eta^{-1}\circ\eta'(x,\gamma)=(\varphi(\gamma)x,\gamma).
\]
Let $\tilde{\gamma}\in\bar{\pi}_1(T'S)$ and write $\eta'^{-1}(\tilde{\gamma})=(x,\gamma)$. Then 
\begin{align*}
\rho(\tilde{\gamma}) &= (\rho_0\circ\eta^{-1}\circ\eta'\circ\eta'^{-1})(\tilde{\gamma}) \\
&= (\rho_0\circ\eta^{-1}\circ\eta')(x,\gamma) \\
&= \rho_0(\varphi(\gamma)x,\gamma) \\
&= \rho_0(\varphi(\gamma),1)\rho_0(x,\gamma).
\end{align*}
Finally, we have 
\begin{align*}
\rho_0(x,\gamma) &= (\rho_0\circ\eta'^{-1})(\tilde{\gamma}) \\
&= \rho'(\tilde{\gamma})
\end{align*}
and therefore 
\[
\rho(\tilde{\gamma})=\sigma(\gamma)\rho'(\tilde{\gamma})
\]
where we have defined $\sigma(\gamma)=\rho_0(\varphi(\gamma),1)$. The representation $\rho^\circ$ changes in exactly the same way, so the pair $(\mathcal{L},\mathcal{L}^\circ)$ of twisted local systems determined by these representations is well defined modulo the diagonal action of the group $\Delta_{SL_m}=\Hom(\pi_1(S),Z(SL_m(\mathbb{C})))$. This completes the Step~2 of the proof.

Now the generic part of the moduli space $\mathcal{D}_{PGL_m,S}$ is an affine scheme over~$\mathbb{Q}$, and one can check that the map constructed above is defined by polynomials with coefficients in~$\mathbb{Q}$. Hence, by Lemma~\ref{lem:pointsfield}, we have a birational map $\mathcal{D}_{\mathbf{i}}\dashrightarrow\mathcal{D}_{PGL_m,S}$.
\end{proof}

Theorem~\ref{thm:main} proves that the rational map defined by Fock and Goncharov is a birational equivalence of the symplectic double moduli space and a torus. Since this torus is open in the full cluster variety, we have proved Theorem~\ref{thm:intromain}.

\section*{Acknowledgments}
\addcontentsline{toc}{section}{Acknowledgements}

I thank Alexander~Goncharov for advising me on this project. I thank Dhruv~Ranganathan, Linhui~Shen, and Daping~Weng for helpful discussions. This work was partially supported by NSF grant DMS-1301776.

\end{document}